\documentclass[a4paper,12pt]{article}
\usepackage{amsthm}
\usepackage{amsfonts}
\usepackage{amsmath}
\usepackage{amssymb}
\usepackage[pdftex]{graphicx}
\usepackage{hyperref}
\usepackage{xcolor}
\newtheorem{thm}{Theorem}[section]
\theoremstyle{definition}
\newtheorem{definition}[thm]{Definition}
\theoremstyle{remark}
\newtheorem{note}{Note}[section]
\theoremstyle{plain}
\newtheorem{lem}[thm]{Lemma}

\newtheorem{proposition}[thm]{Proposition}

\newcommand{\bq}{/\!\!/}
\newcommand{\diag}{\operatorname{diag}}
\newcommand{\g}[1]{\mathfrak{#1}}
\newcommand{\IM}{\operatorname{Im}}

\begin{document}
\title{Cohomogeneity two Bazaikin spaces}
\author{Jason DeVito\footnote{The University of Tennessee at Martin}  \footnote{jdevito1@ut.utm.edu} and Rachel Flores$^*$}
\date{}
\maketitle

\begin{abstract} We study the sectional curvature of all of the cohomogeneity two Bazaikin spaces with respect to a Riemannian metric construction due to Wilking. We show that, in contrast to the cohomogeneity one and homogeneous case, for all of the cohomogeneity two examples, the set of points with strictly positive curvature does not have full measure.
\end{abstract}

\section{Introduction}

Suppose $\overline{q} = (q_1,...,q_5)\in \mathbb{Z}^5$ is a $5$-tuple of odd integers and set $q = \sum_{i=1}^5 q_i$.  Then we may define an action of $Sp(2)\times S^1$ on $SU(5)$ via the formula $$(A+Bj, z)\ast C = \diag(z^{q_1}, z^{q_2}, z^{q_3}, z^{q_4}, z^{q_5}) C \begin{bmatrix} A & B& \\ -\overline{B} & \overline{A}& \\  & & z^q\end{bmatrix}^{-1}.$$  When this action is effectively free, the quotient space is denoted $\mathcal{B}_{\overline{q}}$ and is called a Bazaikin space.

Bazaikin spaces were introduced by Bazaikin \cite{Baz1} where he showed that an infinite sub-family of these spaces admit Riemannian metrics of positive sectional curvature.    In more detail, Bazaikin equipped $SU(5)$ with a particular left invariant Riemannian metric for which the $(Sp(2)\times S^1)$-action is isometric, independent of $\overline{q}$, and showed the induced submersion metric on $\mathcal{B}_{\overline{q}}$ has positive sectional curvature.

Since then, due to the work of \cite{DE,DS,Ke1}, the curvature of Bazaikin spaces with respect to the Bazaikin metric is essentially completely understood.  To describe the previously known results, we recall that a Riemannian manifold is said to be \textit{almost positively curved} if the set of points at which all two-planes are positively curved is open and dense.  A Riemannian manifold is called \textit{quasi-positively curved} if it has non-negative sectional curvature everywhere and a point where all two-planes are positively curved.

\begin{thm}\label{thm:oldknowledge} Suppose a Bazaikin space $\mathcal{B}_{\overline{q}}$ is equipped with the Bazaikin metric.  Then

\begin{enumerate} \item  $\mathcal{B}_{\overline{q}}$ is strictly positively curved  if and only if the sign of $q_{\sigma(1)} + q_{\sigma(2)}$ is independent of $\sigma \in S_5$, the permutation group on $\{1,2,3,4,5\}.$

\item  $\mathcal{B}_{\overline{q}}$ is almost positively curved, but not positively curved everywhere, if and only if $\overline{q}$ is a permutation of $\pm (1,1,1,1,-1).$

\item  $\mathcal{B}_{\overline{q}}$ has a zero-curvature plane at every point if and only if  $\overline{q}$ is a permutation of $\pm(1,1,1,-1,-3).$

\item  $\mathcal{B}_{\overline{q}}$ is quasi-positively curved, but not almost positively curved, if and only if and only if none of the above three cases occur.
\end{enumerate}
\end{thm}

The first case in Theorem \ref{thm:oldknowledge} is due to Dearicott and Eschenburg \cite{DE}, the ``if" part of the second case is due to Kerin \cite{Ke1}, and the remaining statements are due to the first author and Sherman \cite{DS}.

Bazaikin spaces have a presentation as a biquotient: $\mathcal{B}_{\overline{q}}\cong G\bq H$ (see Section \ref{sec:bazaikinspaces}).  Setting $L = N_{G\times G}(H)^0$, the identity component of the normalizer of $H$ in $G\times G$, there is a natural action of $L$ on $\mathcal{B}_{\overline{q}}$.   It is easy to see that $L$ acts isometrically on $\mathcal{B}_{\overline{q}}$ when it is equipped with the Bazaikin metric and it is moreover known that when this action has cohomogeneity at most one, that $L$ is the identity component of the isometry group \cite{GSZ}.  

The case where $L$ acts with cohomogeneity at most one is essentially completely understood.  The unique homogeneous (i.e., cohomogeneity zero) Bazaikin space was originally discovered by Berger \cite{Ber} where he showed it admits a normal homogeneous metric of positive sectional curvature.  In addition, all but one of the cohomogeneity one Bazaikin spaces fall into case 1 of Theorem \ref{thm:oldknowledge} \cite{Wu}.  The remaining cohomogeneity one Bazaikin space  is case 2 of Theorem \ref{thm:oldknowledge}, so it has positive curvature almost everywhere.

The goal of this paper is to continue the study of Bazaikin spaces, focusing on the case where $L$ acts with cohomogeneity two.  As we show in Proposition \ref{prop:free} below, there are infinitely many such Bazaikin spaces and each of them falls into case 4 of Theorem \ref{thm:oldknowledge}.  In particular, one must use a different metric if one wants to equip these cohomogeneity two Bazaikin spaces with an almost positively curved metric.

There is one other known metric construction for Bazaikin spaces, which more generally applies to biquotients.  It was  introduced by Wilking \cite{Wi}, where he constructed several infinite families of almost positively curved manifolds.  Given a Wilking metric on a Bazaikin space $\mathcal{B}_{\overline{q}}$, we define the \textit{natural isometry group} to consist of those elements of $L$ which act by isometries.

Our main result is that cohomogeneity two Bazaikin spaces do not enjoy the same curvature properties as cohomogeneity zero and one Bazaikin spaces.

\begin{thm}\label{thm:main2}  Suppose $\mathcal{B}_{\overline{q}}$ is equipped with a Wilking metric for which the natural isometry group acts with cohomogeneity two.  Then, either $\mathcal{B}_{\overline{q}}$ is the Berger space or $\mathcal{B}_{\overline{q}}$ is not almost positively curved.
\end{thm}

Theorem \ref{thm:main2} indicates that new techniques are required in order to equip the cohomogeneity two Bazaikin spaces with a metric of almost positive curvature.

We now outline the rest of this paper.  In Section \ref{sec:wilking}, we recall the necessary background on Wilking metrics on biquotients, including the determination of the natural isometry group.  In Section \ref{sec:bazaikinspaces}, we first recall the definition and properties of Bazaikin spaces which are independent of the choice of metric.   This includes Proposition \ref{prop:cohom2} which characterizes which $\mathcal{B}_{\overline{q}}$ are cohomogeneity two in terms of the $q_i$ defining it.  We determine the full set of Wilking metrics whose natural isometry group acts via a cohomogeneity two action in Section \ref{sec:bazmetrics}, and show that the verification of Theorem \ref{thm:main2} reduces to the verification of one particular Wilking metric in Proposition \ref{prop:reducetomain}. 
 Finally, in Section \ref{sec:zero}, we prove Theorem \ref{thm:main2}.

Both authors were partially supported by NSF DMS 2405266.  Flores is also partially supported by the Bill and Roberta Blankenship Undergraduate
Research endowment. They are grateful for the support.

\section{Metrics on biquotients}\label{sec:wilking}

In this section, we describe the construction of Wilking metrics on biquotients.  We begin with the definition of a biquotient.

\begin{definition} A manifold is said to be a biquotient if it is diffeomorphic to an orbit space $G\bq H$, where $G$ is a compact Lie group, $H\subseteq G\times G$ is a closed subgroup, and $H$ acts effectively freely on $G$ via the formula $(h_1,h_2)\ast g = h_1 g h_2^{-1}$.  If $H = H_1\times H_2$ with each $H_i\subseteq G$, we sometimes denote the biquotient as $H_1\backslash G/H_2$.
\end{definition}

Biquotients are natural generalizations of homogeneous spaces, which arise when either $H_1$ or $H_2$ is trivial.  Each Bazaikin space is a biquotient of the form $SU(5)\bq (S^1\times Sp(2))$, see Section \ref{sec:bazaikinspaces} for more details.

\bigskip

We now describe two metric constructions on Lie groups which will  eventually be used to define metrics on biquotients as submersion metrics.  

We begin with the notion of a Cheeger deformation.  Suppose $G$ is a compact Lie group and suppose that $K\subseteq G$ is a closed subgroup.  Then there is a $K$-action on $G\times K$ given by $k\ast (g,k_1) = (gk^{-1}, k k_1)$.  This action is free and the quotient is diffeomorphic to $G$ with a diffeomorphism being induced from the map $G\times K\rightarrow G$ with $(g,k_1)\mapsto gk_1$.

If $\langle \cdot,\cdot\rangle_0$ is a left $G$-invariant and right $K$-invariant metric on $G$, then we may equip $G\times K$ with the metric $\langle \cdot, \cdot \rangle_0 + t \langle \cdot, \cdot\rangle_0|_K$ with $t\in (0,\infty)$, a fixed parameter.  Since the above $K$-action is then isometric, $G$ inherits a submersion metric $\langle \cdot, \cdot \rangle_K$, called the \textit{Cheeger deformation} of $\langle \cdot, \cdot\rangle_0$ in the direction of $K$.  From the Gray-O'Neill formulas \cite{Gr,On1}, $\langle \cdot, \cdot \rangle_K$ has non-negative sectional curvature if $\langle \cdot,\cdot \rangle_0$ has non-negative sectional curvature.  Furthermore, the left multiplication of $G$ on the first factor of $G\times K$ and right multiplication of $K$ on the second factor of $G\times K$ are both isometric, and descend to $(G,\langle \cdot, \cdot\rangle_K)$.  In particular, $\langle \cdot, \cdot\rangle_K$ is both left $G$-invariant and right $K$-invariant.

If $K'\subseteq K$, then we observe that we can Cheeger deform $\langle \cdot, \cdot\rangle_K$ in the direction of $K'$.  More generally, given a chain of closed subgroups $\{e\} = K_{n+1}\subseteq K_n\subseteq ...\subseteq K_1\subseteq K_0 =G$ one may obtain an iterated Cheeger deformed metric $\langle \cdot, \cdot \rangle_{K_n\subseteq ... \subseteq K_1}$ of $\langle \cdot, \cdot, \rangle_0$.  If the initial metric $\langle \cdot, \cdot \rangle_0$ is non-negatively curved and left $G$-invariant, then the same is true of $\langle \cdot, \cdot \rangle_{K_n\subseteq... \subseteq K_1}$. We will always work under the assumption that $\langle \cdot, \cdot\rangle_0$ is bi-invariant, so every Cheeger deformed metric we consider is non-negatively curved and left invariant.

Being left invariant, such a Cheeger deformed metric is completely determined by its value at the identity $e$, an inner product on $T_e G\cong \mathfrak{g}$.  To describe this inner product, we let $\mathfrak{k}_i$ denote the Lie algebra of $K_i$ and, for $i$ from $1$ to $n+1$, we define $\mathfrak{p}_i  = \mathfrak{k}_i^\bot\cap \mathfrak{k}_{i-1}$ where the orthogonal complement is with respect to $\langle \cdot,\cdot\rangle_0$.  Thus, we have a decomposition $\mathfrak{g} = \bigoplus_i \mathfrak{p}_i$ and we write $X\in \mathfrak{g}$ as $X =  \sum X_i$ with $X_i\in \mathfrak{p}_i$.  Then, from \cite{Di}, there is a sequence of real numbers $1 = \sigma_0 > \sigma_1 > ... >\sigma_{n+1} > 0$ with the property that $\langle X, Y\rangle_{K_n\subseteq ... \subseteq K_1} = \langle \phi(X), Y\rangle_0$ where $\phi(X) = \sum_{i=1}^{n+1} \sigma_{i-1} X_i$.

\begin{proposition}\label{prop:isom} Suppose $\langle \cdot,\cdot \rangle_{K_n\subseteq ...\subseteq K_1}$ is a metric on $G$ obtained as an iterated Cheeger deformation of a bi-invariant metric $\langle \cdot, \cdot \rangle_0$ via a chain of closed subgroups $K_n\subseteq ...\subseteq K_1 \subseteq G$.  Then the right multiplication action by a closed subgroup $L\subseteq G$ is isometric if and only if $L$ normalizes all $K_i$.

\end{proposition}

\begin{proof} Assume that $L$ normalizes each $K_i$ and let $\ell\in L$.  Then the adjoint map $Ad(\ell):\mathfrak{g}\rightarrow\mathfrak{g}$ preserves each $\mathfrak{k}_i$ and, because $Ad(\ell)$ an isometry with respect to $\langle \cdot, \cdot \rangle_0$, $Ad(\ell)$ preserves each $\mathfrak{p}_i$.  In addition, since $\langle \cdot,\cdot\rangle_{K_n\subseteq ... \subseteq K_1}$ is a sum of the restrictions of bi-invariant metrics on $\mathfrak{p}_i$, $Ad(\ell)$ is an $\langle\cdot,\cdot\rangle_{K_n\subseteq...\subseteq K_1}$-isometry on each factor.  It follows that $Ad(\ell)$ is an isometry with respect to $\langle \cdot, \cdot\rangle_{K_n\subseteq ... \subseteq K_1}$.  Since this metric is left invariant, it follows that right multiplication by $\ell^{-1}$ is an isometry.

Conversely, assume that for $\ell\in L$, right multiplication by $\ell^{-1}$ is an isometry of $\langle \cdot,\cdot\rangle_{K_n\subseteq ... \subseteq K_1}$.  Since the metric is left invariant, it follows that $Ad(\ell)$ is an isometry of $\langle\cdot, \cdot, \rangle_{K_n\subseteq ...\subseteq K_1}$.  We claim that $Ad(\ell)\circ \phi  = \phi \circ Ad(\ell)$.  To see this, we note that for any $X,Y\in \mathfrak{g}$, that \begin{align*} \langle Ad(\ell) \phi(X),Ad(\ell)Y\rangle_0 &= \langle \phi(X),Y\rangle_0\\ &= \langle X,Y\rangle_{K_n\subseteq... \subseteq K_1}\\ &= \langle Ad(\ell) X,Ad(\ell)Y\rangle_{K_n\subseteq ... \subseteq K_1} \\ &= \langle\phi(Ad(\ell) X), Ad(\ell)Y\rangle_0.\end{align*}  Thus, $\langle Ad(\ell)\phi(X),Ad(\ell) Y\rangle_0 = \langle \phi(Ad(\ell) X), Ad(\ell) Y\rangle_0$ for all $X$ and $Y$.  Since $Ad(\ell)Y$ varies over $\mathfrak{g}$ as $Y$ varies over $\mathfrak{g}$, this implies that $Ad(\ell)\circ \phi = \phi \circ Ad(\ell)$.

Since $Ad(\ell)$ and $\phi$ commute, $Ad(\ell)$ must preserve the eigenspaces of $\phi$.  However, by inspection, $\phi$ has the $\mathfrak{p}_i$ as eigenspaces with distinct eigenvalues $\sigma_{n+1} < \sigma_n < ... < \sigma_1 = 1$.  It follows that $Ad(\ell)$ preserves each $\mathfrak{p}_i$.  From this it follows that $Ad(\ell)$ preserves each $\mathfrak{k}_i$, which implies that $L$ normalizes each $K_i$.
\end{proof}

The curvature of iterated Cheeger deformations of bi-invariant metrics is well understood \cite{Di}.

\begin{proposition}\label{prop:curvature}  Suppose $\langle \cdot, \cdot\rangle_{K_n\subseteq ... \subseteq K_1}$ is an iterated Cheeger deformation of a bi-invariant metric $\langle \cdot,\cdot\rangle_0$.  Then a plane $\sigma = \operatorname{span}\{X,Y\}$ at the identity has zero-curvature with respect to $\langle\cdot,\cdot\rangle_{K_1\subseteq ... \subseteq K_n}$ if and only if for each $j$ from $0$ to $n$, we have  $[\phi(X)_{\mathfrak{k}_j},\phi(Y)_{\mathfrak{k}_j}] = 0$.

\end{proposition}

\bigskip

We now describe Wilking metrics.  Suppose $G$ is given two iterated Cheeger deformed metrics (with the chain-of-subgroups notation suppressed) $\langle \cdot, \cdot\rangle_{\ell}$ and $\langle \cdot, \cdot \rangle_r$ and assume that both of these metrics have have non-negative sectional curvature.  We may then form the product Riemannian manifold $(G\times G, \langle \cdot, \cdot\rangle_{\ell} + \langle \cdot, \cdot\rangle_{r})$, which again has non-negative sectional curvature.  Since both $\langle \cdot, \cdot\rangle_{\ell}$ and $\langle \cdot , \cdot \rangle_r$ are left $G$-invariant, the diagonal $G$-action given by $g\ast (g_1,g_2) = (gg_1, gg_2)$ is isometric and free.  The quotient $\Delta G\backslash (G\times G)$ is diffeomorphic to $G$, with a diffeomorphism induced from $(g_1,g_2)\mapsto g_1^{-1}g_2$.  Thus, we may equip $G$ with the induced submersion metric $\langle \cdot, \cdot \rangle_1$, which is again non-negatively curved.  Observe that if $\langle \cdot,\cdot \rangle_{\ell}$ and $\langle \cdot,\cdot \rangle_r$ are invariant under right multiplication by closed subgroups $K_{\ell}$ and $K_r\subseteq G$, respectively, then  right multiplication by $K_\ell\times K_r$ on $G\times G$ descends to an isometric action of $K_{\ell}\times K_r$ on $(G,\langle \cdot, \cdot \rangle_1)$, so that the Wilking metric is left $K_\ell$-invariant and right $K_r$-invariant.

Now, if a closed subgroup $H\subseteq K_\ell\times K_r$ is chosen so that the biquotient $H$-action on $G$ given by $(h_1,h_2)\ast g = h_1 g h_2^{-1}$ is free, then the orbit space inherits a non-negatively curved metric from $G$, called a \textit{Wilking metric}.   More precisely, we make the following definition.

\begin{definition} A \textit{Wilking metric} on a biquotient of the form  $\Delta G \backslash (G\times G)/ H$ is any metric obtained as a  submersion metric where $G\times G$ is equipped with a right $H$-invariant metric given as a sum of left $G$-invariant metrics, each obtained as an iterated Cheeger deformation of a fixed bi-invariant metric $\langle \cdot, \cdot \rangle_0$ on $G$.
\end{definition}

If $L\subseteq G\times G$ and $H$ is normal in $L$, that is, $L\subseteq N_{G\times G}(H)$, then the action of $L$ on $G\times G$ given by right multiplication descends to a well-defined action by $L/H$ on $\Delta G\backslash(G\times G)/H$.  This $L/H$ action on $\Delta G\backslash (G\times G)/H$ is often isometric.  For example, this is the case if the $L$-action on $G\times G$ given by right multiplication is isometric.

\begin{definition}  The \textit{natural isometry group} of  Wilking metric on $\Delta G\backslash (G\times G)/H$ is the largest subgroup of $N_{G\times G}(H)/H$ that acts by isometries.

\end{definition}

 The curvature of Wilking metrics is also well understood.  To describe the result, we first establish notation.  First, given $g\in G$ and $X\in \mathfrak{g}$, we define $$\hat{X}:=( -\phi_{\ell}^{-1}(Ad_{g^{-1}}(X)), \phi_r^{-1}(X))\in \g{g}\oplus \g{g}$$ where $\phi_{\ell}$ and $\phi_r$ are the metric tensors relating the metric on the left and right factors of $G$ to the bi-invariant metric.  For the projection $G\times G\rightarrow \Delta G\backslash (G\times G)/H$, we let $\mathcal{H}_{(g,e)}$ denote left-translation of horizontal space at $(g,e)$ to the identity $(e,e)$. From, e.g., \cite[Equation (9)]{Ke2}, we have \begin{equation}\label{eqn:horizontal}\mathcal{H}_{(g,e)} = \{\hat{X}: \langle X, Ad_g H_1 - H_2\rangle_0 = 0 \text{ for all }(H_1,H_2)\in \g{h}\subseteq \g{g}\oplus \g{g}\}.\end{equation}    We now have  the following proposition.

 \begin{proposition}\label{prop:zerowilking} Suppose $[(g,e)]\in \Delta G\backslash (G\times G)/H$.  Then, there is a zero-curvature plane at $[(g,e)]$ if and only if there are linearly independent vectors $X,Y\in \mathfrak{g}$ satisfying each of the following conditions.

 \begin{enumerate}\item  $\langle X, Ad_g H_1 - H_2\rangle_0 = \langle Y, Ad_g H_1 - H_2\rangle_0 =0$ for all $(H_1,H_2)\in \mathfrak{h}\subseteq \mathfrak{g}\oplus \mathfrak{g}$

 \item  With respect to the metric on the second factor of $G\times G$, the plane $\sigma_2:=\operatorname{span}\{\phi_r^{-1} X, \phi_r^{-1} Y\}$ has zero-curvature .

 \item  With respect to the metric on the first factor of $G\times G$, the plane $\sigma_1:=\operatorname{span}\{\phi_{\ell}^{-1} (Ad_{g^{-1}} (X)), \phi_{\ell}^{-1}(Ad_{g^{-1}}(Y))\}$ has zero-curvature.

 \end{enumerate}

 \end{proposition}

 \begin{proof}Assume first that there is a zero-curvature plane at the point $[g,e]$.  Then, from the Gray-O'Neill formulas \cite{Gr,On1}, the horizontal lift $\hat{\sigma}$ of this plane to $G\times G$ must have zero curvature.

Being horizontal, the left-translation of $\hat{\sigma}$ to the identity is spanned by a pair of linearly independent vectors $\hat{X},\hat{Y}\in \mathcal{H}_{(g,1)}$ defined by $X,Y\in \g{g}$.  From Equation \eqref{eqn:horizontal}, we deduce that condition $1$ of the Proposition holds.

 Because the metric on $G\times G$ is a product of non-negatively curved metrics, this plane has zero-curvature if and only if both of its projections to each factor of $G\times G$ has zero-sectional curvature.  This yields the second and third conditions of this proposition.

 \bigskip

 Conversely, assume there are $X,Y\in \mathfrak{g}$ satisfying each of the three listed conditions.  Then, the left translates $\hat{X}$ and $\hat{Y}$ to $[(g,e)]$ span a horizontal zero-curvature plane.  But Tapp \cite{Ta2} has shown that in this situation, the projection of a horizontal zero-curvature plane has zero-curvature.

 \end{proof}

 \section{Bazaikin spaces}

\subsection{Background on Bazaikin spaces}\label{sec:bazaikinspaces}

Consider the Lie groups $G = SU(5)$.  Given a $5$-tuple of odd integers $\overline{q} = (q_1,q_2,q_3,q_4,q_5)\in \mathbb{Z}^5$  with $\gcd(q_1,q_2,q_3,q_4,q_5) = 1$, we obtain an action of $H:=Sp(2)\times S^1$ on $G$ by the formula $$(A+Bj, z)\ast C =  \diag(z^{q_1}, z^{q_2}, z^{q_3}, z^{q_4}, z^{q_5}) C \begin{bmatrix} A & B & \\ -\overline{B} & \overline{A} & \\ & & z^{q}\end{bmatrix}^{-1},$$ where $q := q_1 + q_2 + q_3 + q_4 + q_5$.

It is well known that this action is effectively free if and only if $$\gcd(q_{\sigma(1)} + q_{\sigma(2)}, q_{\sigma(3)} + q_{\sigma(4)}) = 2$$ for all $\sigma \in S_5$, the permutation group on the set $\{1,2,3,4,5\}$.  Moreover, when the action is effectively free, the kernel of the action is $(-I,-1)$, so that $H/(-I,-1)$ acts freely.  In addition, when the action is effectively free, the quotient space is a smooth manifold called a \textit{Bazaikin space} and is denoted  by $\mathcal{B}_{\overline{q}}$.  Bazaikin spaces were introduced by Bazaikin \cite{Baz1}, where he showed that an infinite subfamily of them admit Riemannian metrics of positive sectional curvature.

We observe that $$\det\left(\diag(z^{q_1}, z^{q_2}, z^{q_3}, z^{q_4}, z^{q_5}\right) = \det\left( \begin{bmatrix} A &  B & \\ -\overline{B} & A &\\ & & z^q\end{bmatrix}\right) = z^q$$ which is, in general, not equal to $1$.  In particular, the action $\ast$ is not a biquotient action.  As such, it will be convenient to replace $\ast$ with another $H$-action $\bullet$, defined by $(A+Bj,z)\bullet C =$ $$ \diag(z^{5q_1-q}, z^{5q_2-q}, z^{5q_3 - q}, z^{5q_4 - q}, z^{5q_5-q}) C \begin{bmatrix} z^{-q} A & z^{-q} B & \\ -z^{-q} \overline{B} & z^{-q} \overline{A} & \\  & & z^{4q}\end{bmatrix}^{-1}.$$  We note that $\bullet$ is always an ineffective action since $(-I,-1)$ acts trivially.  In addition, we observe that $$\det\left(\diag(z^{5q_1-q}, z^{5q_2-q}, z^{5q_3 - q}, z^{5q_4 - q}, z^{5q_5-q}\right) =  1$$ and $$\det\left(\begin{bmatrix} z^{-q} A & z^{-q} B & \\ -z^{-q} \overline{B} & z^{-q} \overline{A} & \\  & & z^{4q}\end{bmatrix}\right) = 1$$ so that both of these matrices are elements of $G$.  In particular, the $\bullet$-action is a biquotient action.

The relationship between $\ast$ and $\bullet$ is given by the following proposition.

\begin{proposition} For any $(A+Bj,z)\in H$ and any $C\in G$, we have $$(A+Bj,z)\bullet C = (A+Bj, z^5) \ast C.$$  In particular, the $\ast$-orbits coincide with the $\bullet$-orbits and the action by $\ast$ is effectively free if and only if the action by $\bullet$ is effectively free.\end{proposition}

\begin{proof}  To verify the proposition it is sufficient to verify that the $\bullet$-action is the same as the $\ast$-action, except that the $S^1$ factor acts at $5$-times the speed.  To that end, we begin with the action $\ast$ and replace each $q_i$ and $q$ by $5q_i$ and $5q$, respectively to obtain a new action $\ast'$.  Obviously, $\ast'$ agrees with $\ast$ except that the $S^1$ factor acts at $5$-times speed.  Simultaneously replacing  $\diag(z^{5q_1}, z^{5q_2}, z^{5q_3}, z^{5q_4}, z^{5q_5})$ with $z^{-q} \diag (z^{5q_1}, z^{5q_2}, z^{5q_3}, z^{5q_4}, z^{5q_5})$ and $$\begin{bmatrix} A & B & \\ -\overline{B} & \overline{A} & \\ & & z^{5q} \end{bmatrix}\text{ with }z^{-q}\begin{bmatrix} A & B & \\ -\overline{B} & \overline{A} & \\ & & z^{5q} \end{bmatrix},$$ we obtain $\bullet$.  However, because $z^{-q}I$ commutes with all $5\times 5$ matrices, it follows that $\bullet = \ast'$, thus completing the proof.

\end{proof}

As mentioned above, every biquotient $G\bq H$ has a natural action by the normalizer $N_{G\times G}(H)$.  In the context of Bazaikin spaces, we let $H'\subseteq G\times G$ denote the image of the map $H\rightarrow G\times G$ given by $(A+Bj,z)\mapsto$ $$\left(\diag(z^{5q_1-q},z^{5q_2-q}, z^{5q_3-q}, z^{5q_4 - q}, z^{5q_5-5}), \begin{bmatrix}  z^{-q} A & z^{-q}B & \\ -z^{-q}\overline{B} & z^{-q}\overline{A} & \\ &  &z^{4q}\end{bmatrix}\right),$$ so that $\mathcal{B}_{\overline{q}}\cong G\bq H'$.  Let $L:=N_{G\times G}(H')^0$ be the identity component of the normalizer .

We wish to describe the structure of $L$.  To that end, to each $\overline{q} = (q_1,q_2,q_3,q_4,q_5)$ we associate a partition of $5$ where each part of the partition counts the number of repeated $q_i$ in $\overline{q}$.  For example, the partition $2+3$ corresponds to the $5$-tuple $\overline{q}$ with $q_1 = q_2 \neq q_3 =  q_4 = q_5$ (up to permutation), while the partition $1+1+1+1+1$ corresponds to the case where all $q_i$ are distinct.  Given such a partition $5 = k_1 + k_2 + ... + k_n$, we let $S:=$  $$ S(U(k_1)\times U(k_2)\times ...\times U(k_n)) = \{\diag(A_1,A_2,...,A_n)\in SU(5): A_i\in U(k_i)\}.$$

\begin{proposition}  Suppose $\overline{q} = (q_1,q_2,q_3,q_4,q_5)$ and let $5 = k_1 + k_2  +...+k_n$ denote the corresponding partition of $5$.  Then $$L = (S\times \{I\})\cdot H'.$$
\end{proposition}

\begin{proof}For $i = 1,2$, let $p_i:G\times G\rightarrow G$ denote the projection onto the $i$-th factor.  It is clear that $S$ centralizes $p_1(H')$ and that $H'\subseteq L$, so we conclude that $(S\times  \{I\})\cdot H'\subseteq L.$

To see the reverse inclusion, first observe that since $L$ normalizes $H'$, $p_1(L)$ is a connected group normalizing $p_1(H')$.  As $p_1(H')$ is a circle, the identity component of its automorphism group is trivial. 
 Thus, conjugation by $p_1(L)$ must be trivial so that $p_1(L)$ centralizes $p_1(H')$.  In addition, $p_2(L)$ normalizes $p_2(H')$.  From \cite[Table A]{WiZi}, we see that $p_2(H')$ is self-normalizing.  Thus $p_2(L) = p_2(H')$.  Now, given $(\ell_1,\ell_2)\in L$, it follows that there is an element of the form $(h,\ell_2)\in H'$.  The element $(\ell_1 h^{-1}, I)\in L$ must normalize $H'$, which implies that $p_1(\ell_1 h^{-1},I) = \ell_1 h^{-1}$ centralizes $p_1(H')$.  Thus, $\ell_1 h^{-1}$ maps each eigenspace of any element of $p_1(H')$ to itself, which clearly implies that $\ell_1 h^{-1}\in S$.  It now easily follows that $L\subseteq (S\times \{I\})\cdot H'.$

\end{proof}

Since the action of $H'$ on $\mathcal{B}_{\overline{q}}$ is obviously trivial, the $L$-action on $\mathcal{B}_{\overline{q}}$ descends to an action by $L/H'$.  Since $S$ has trivial projection to the second factor of $G\times G$, it follows that $S \times\{I\}$ has finite intersection with $H'$.  From this it follows that up to finite kernel, the action of $L/H'$ on $\mathcal{B}_{\overline{q}}$ is isomorphic to the action induced via left multiplication by $S$.  In addition, it follows that the $S$-action has finite kernel.  We now determine when the cohomogeneity of this $S$-action is two.

\begin{proposition}\label{prop:cohom2}  Suppose $\overline{q} = (q_1,q_2,q_3,q_4,q_5)$ and that the $S$-action on $\mathcal{B}_{\overline{q}}$ has cohomogeneity exactly two.  Then $\overline{q}$ must, up to permutation, be of the form $(q_1,q_1,q_1,q_4,q_4)$ for a pair of distinct $q_1\neq q_4\in \mathbb{Z}$.
\end{proposition}

\begin{proof}  We will first show that if at least three $q_i$ are distinct, then the $S$-action has cohomogeneity three or larger.  Since each Bazaikin space is $13$-dimensional and $S$ acts almost effectively, it suffices to show that $\dim S \leq 10$ if three $q_i$ are distinct.  If all $5$ are distinct, $\dim S = 4$ since $S =   S(U(1)^5) \cong T^4$.  If $4$ are distinct, $\dim S = \dim S(U(2)\times U(1)^3) = 6$.  If there are precisely three distinct $q_i$, then up to reordering the $q_i$, there are two cases:  $q_1,q_1,q_1, q_2,q_3$ and $q_1,q_1, q_2, q_2, q_3$.  Then $S$ is equal to $S(U(3)\times U(1)^2)$ or $S(U(2)^2\times U(1))$, so has dimension $10$ and $8$ respectively, completing the proof in this case.

Thus, we may assume that there are at most two distinct $q_i$s appearing.  If all $q_i$ are equal, then $S = SU(5)$ acts transitively.  So, we may assume there are precisely two distinct numbers among the $q_i$.  It follows that, up to order, we must have either the form $(q_1,q_1,q_1,q_1,q_5)$ or the form $(q_1,q_1,q_1,q_4,q_4)$.  Since the first form is well known to consist of cohomogeneity one actions \cite{GSZ}, the proof is now complete.

\end{proof}

We still need to demonstrate that there are $5$-tuples $\overline{q} = (q_1,q_1,q_1,q_5,q_5)$ satisfying the $\gcd$ conditions necessary for having an effectively free action of $H'$ on $G$, and we need to compute the cohomogeneity of the $S$-action on these space.  The first step is contained in the following proposition.

\begin{proposition}\label{prop:free}  Suppose $q_1 = q_2 = q_3$ and $q_4=q_5$.  Then the above action is effectively free if and only if $q_1 + q_4 \in \{\pm 2\}$.  In particular, there are infinitely many such free actions.
\end{proposition}

\begin{proof}  Assume initially that $q_1$ and $q_4$ are chosen so that the action is free.  Then we find that $$2 = \gcd(q_1 + q_4, q_2 + q_5) = \gcd(q_1 + q_4, q_1 + q_4) = |q_1 + q_4|.$$  The condition $|q_1 + q_4| = 2$ is clearly equivalent $q_1 + q_4 \in \{\pm 2\}$.

Conversely, assume that $q_1 + q_4\in \{\pm 2\}$.  We need to check that $\gcd(q_{\sigma(1)} + q_{\sigma(2)}, q_{\sigma(3)} + q_{\sigma(4)}) = 2$ for all $\sigma \in S_5$.  Observe that since $q_4$ only appears twice in the set $\{q_1,q_2,q_3,q_4,q_5\}$, a choice of four elements consists of $2$ or $3$ copies of $q_1$.

If three copies of $q_1$ are chosen, we obtain $$\gcd(q_1 + q_1, q_1 + q_4) = \gcd(2q_1, 2) = 2\gcd(q_1,1) = 2,$$ as needed.

If two copies of $q_1$ are chosen and two copies of $q_4$ are chosen, then either the two copies of $q_1$ are paired up, or not.  If they are paired, we obtain $\gcd(q_1 + q_1, q_4 +  q_4) = 2\gcd(q_1,q_4) = 2\gcd(q_1, \pm 2- q_1) = 2\gcd(q_1, 2) = 2$ since $q_1$ is odd.  If they are not paired, we obtain $\gcd(q_1 + q_4, q_1 + q_4) = \gcd(2,2) = 2$.
\end{proof}

\begin{note}\label{note:case4}
The conditions $|q_1 + q_4| = 2$ and $q_1\neq q_4$ imply that $q_1$ and $q_4$ have opposite signs.  In particular, $q_1 + q_2 = 2q_1$ and $q_4+q_5 = 2q_4$ have opposite signs so that none of these Bazaikin spaces falls into case 1 of Theorem \ref{thm:oldknowledge}.  It is obvious that neither case $2$ nor case $3$ apply.  Hence, Theorem \ref{thm:oldknowledge} implies that each of the Bazaikin spaces is quasi-positively curved but not almost positively curved when equipped with the Bazaikin metric.
\end{note}

We now focus on second step - determining the cohomogeneity of the $S=(S(U(3)\times U(2))$-action on $\mathcal{B}_{\overline{q}}$.

  We let $\mathcal{F}$ denote the set $$\left\{\begin{bmatrix} 0 & -1 & 0 & 0 & 0\\0 & 0 & \cos(\alpha) & -\sin(\alpha) & 0  \\ \cos(\theta) & 0 & 0 & 0 & -\sin(\theta) \\ \sin(\theta) & 0 & 0 & 0 & \cos(\theta)\\ 0 & 0 & \sin(\alpha) & \cos(\alpha) & 0 \end{bmatrix} \in SU(5): \theta,\alpha\in [0,\pi/2]\right\}.$$  Observe that $\mathcal{F}$ is two-dimensional, being homeomorphic to a closed disk.  Thus, the following proposition demonstrates that the $S$-action has cohomogeneity at most two.

\begin{proposition}\label{prop:F}  Suppose $A = (a_{ij})\in SU(5)$.  Then, under the above $S$-action on $\mathcal{B}_{\overline{q}}$, the orbit through $[A]$ contains a point $[B]$ with $B\in \mathcal{F}$.
\end{proposition}

\begin{proof} We consider the action by $S \times Sp(2)$ on $SU(5)$ where $S$ acts by left multiplication and $Sp(2)$ acts via right multiplication by inverses.  We will show that the $(S\times Sp(2))$-orbit through any $A\in SU(5)$ intersects $\mathcal{F}$.  Clearly, this will be sufficient to establish the proposition.

To begin with, we restrict the action to the $SU(2)$ factor in $S$.  If we focus on the action of this $SU(2)$ on the sub-column $\begin{bmatrix} a_{45}\\ a_{55}\end{bmatrix} $, we see this is the standard action on $\mathbb{C}^2$.  This action is well known to be transitive on spheres centered on the origin.  Thus, we see that every $S$-orbit passes through a point with $\begin{bmatrix} a_{45} \\ a_{55}\end{bmatrix} = \begin{bmatrix} \cos(\theta) \\ 0 \end{bmatrix}$, where $\theta$ is chosen so that $\cos(\theta) = \sqrt{|a_{45}|^2 + |a_{55}|^2}\geq 0$.  In particular, we may assume that $\theta \in [-\pi/2,\pi/2]$.

We now restrict the action to $Sp(2)$.  Observe that the action by $Sp(2)$ is trivial on the last column, so will preserve the form we achieved in the previous paragraph.  Focus on how $Sp(2)$ acts on the sub-row $(a_{41}, a_{42}, a_{43}, a_{44})$.  This is the standard action of $Sp(2)$ on $\mathbb{H}^2\cong \mathbb{C}^4$.  As this action is also transitive on the unit sphere, we see that the orbit contains a point where $a_{41}$ is real and non-negative and where $a_{42} = a_{43} = a_{44} = 0$.  Since the fourth row must have unit length, and the last entry is $\cos(\theta)$, we find $a_{41} = \sin(\theta) \geq 0$.  In particular, we now have $\theta \in [0,\pi/2]$.

In summary, up to this point, we have the fourth row of $A$ in the form $(\sin(\theta),0,0,0,\cos(\theta))$ and the fifth row has a $0$ in the last slot.  Since the $4$th and $5$th rows are orthogonal, we have $a_{51} = 0$ unless $\theta = 0$.  But if $\theta = 0$, then we may use the $Sp(2)$ action on the fifth row to get $a_{51} = 0$.

Next, we left multiply by the element $\diag(a_{53}/|a_{53}|, 0, 0, 0, \overline{a}_{53}/|a_{53}|)\in S$ to make $a_{53}$ real and non-negative, while keeping the form of the fourth row.

We now further restrict the action to $Sp(1)\subseteq Sp(2)$, embedded in the bottom right corner of $Sp(2)$.  Because of the zeros we already have in the fourth row, the $Sp(1)$-action preserves the form we have already achieved.  Writing the bottom right entry of an element of $Sp(1)$ as $a + bj$, it is easy to see that the $Sp(1)$-action on $(0, a_{52}, a_{53}, a_{54})$ takes this point to $(0, \overline{a} a_{52} + \overline{b} a_{54}, a_{53}, -b a_{52} + a a_{54})$.  In particular, if $(a,b)$ is chosen to be complex-orthogonal to $(a_{52},a_{54})$ and of unit length, we can make $a_{52} = 0$.  By varying the phase of our choice of $(\overline{a},\overline{b})$, we may assume that $a_{54}$ is non-negative.

At this point, the fifth row now has the form $(0,0,a_{53}, a_{54},0)$ with both $a_{53}$ and $a_{54}$ non-negative.  As $|a_{53}|^2 + |a_{54}|^2 = 1$, we may write $a_{53} = \sin(\alpha)$ and $a_{54} = \cos(\alpha)$ for some unique $\alpha \in [0,\pi/2]$.

We have now shown that the $(S\times Sp(2))$-orbit through any matrix in $SU(5)$ contains a matrix whose last two rows have the form $$\begin{bmatrix} \sin(\theta) & 0 & 0 & 0 &\cos(\theta)\\
0 & 0 & \sin(\alpha) & \cos(\alpha) & 0\end{bmatrix}.$$

Now we recall that for $SU(3)\subseteq S$  the quotient by the usual $SU(3)$-action on $SU(5)$ is the complex Stiefel manifold of orthonormal $2$-frames in $\mathbb{C}^5$.  The map which takes a matrix in $SU(5)$ to the set of the last two rows is $SU(3)$-invariant.  As such, it follows that any two matrices with the same last two rows are in the same $SU(3)$-orbit.  It is easy to verify that the matrix given in the proposition lies in $SU(5)$.  Hence, using an appropriate element of $SU(3)$, we may finally move our matrix to one in the form of the proposition.

\end{proof}

The following proposition serves two purposes.  First, it will give us a way of identifying when two matrices in $\mathcal{B}_{\overline{q}}$ are equivalent under the $S$-action.  Second, it will imply that the cohomogeneity of the $S$-action on $\mathcal{B}_{\overline{q}}$ is exactly two.  To state the proposition, given a matrix $A=(a_{ij})\in SU(5)$, we let $\nu(A) =a_{43}a_{51} + a_{44}a_{52} - a_{41}a_{53}-a_{42}a_{54}$. 

\begin{proposition}\label{prop:ABequivalent}  Given $A = (a_{ij}), B = (b_{ij})\in SU(5)$, $[A]$ and $[B]\in \mathcal{B}_{\overline{q}}$ are in the same $S$-orbit if and only if either $|a_{45}|^2 + |a_{55}|^2 = |b_{45}|^2 + |b_{55}|^2 =1$ or both $|a_{45}|^2 + |a_{55}|^2 = |b_{45}|^2 + |b_{55}|^2$ and $|\nu(A)| = |\nu(B)|$.
\end{proposition}

\begin{proof}

We begin by showing that the quantities  $|a_{45}|^2+|a_{55}|^2$ and $|\nu(A)|$ are invariants of the $Sp(2)$-action.  One easily computes that both $a_{45}$ and $a_{55}$ are individually invariant under the $Sp(2)$-action.  In addition, the action of $Sp(2)$ on the two row vectors $(a_{41},a_{42}, a_{43},a_{44})$ and $(a_{51}, a_{52}, a_{53}, a_{53})$ is equivalent to the usual $Sp(2)$ action on $\mathbb{H}^2$ under the usual identification $\mathbb{H}^2\cong \mathbb{C}^4$.  The quantity $\nu(A)$ is precisely the $j$ and $k$ components of the quaternionic inner product of these two rows, so it is preserved by $Sp(2)$.

We now turn attention to the action by $S=S(U(3)\times U(2))$.  Of course, entries in the $U(3)$ factor leave the fourth and fifth rows of $A$ unaltered,  So it is sufficient to restrict attention to the action by $U(2)$ on these last two rows.

Since the action by $U(2)$ preserves lengths, $|a_{45}|^2 + |a_{55}|^2$ is invariant under this action.  For the other quantity, one simply computes that for any $\begin{bmatrix} a & b\\ -\overline{b}z & \overline{a}z\end{bmatrix} \in U(2)$ (where $z\in S^1$ and $|a|^2 + |b|^2 = 1)$ that $$\nu\left(\begin{bmatrix} a & b\\ -\overline{b}z & \overline{a}z\end{bmatrix} A\right) = z\nu(A),$$ so that $|\nu(A)|$ is invariant.  This completes the proof that $|a_{45}|^2 + |a_{55}|^2$ and $\nu(A)$ are invariants under the $S(U(3)\times U(2))\times Sp(2)$-action.  It follows from this that if $A$ and $B$ are orbit equivalent, then $|a_{45}|^2 + |a_{55}|^2 = |b_{45}|^2 + |b_{55}|^2$ and $|\nu(A)|=|\nu(B)|$.

To prove the converse, we break into cases depending on whether $|a_{45}|^2 + |a_{55}|^2 = 1$.  So, assume first that $|a_{45}|^2 + |a_{55}|^2 = |b_{45}|^2 + |b_{55}|^2 = 1$.  Then from Proposition \ref{prop:F}, both $[A]$ and $[B]$ are orbit equivalent to a matrix in $\mathcal{F}$ whose fourth row is $(0,0,0,0,1)$.  As in the proof of Proposition \ref{prop:F}, one can then easily find an element $C\in Sp(2)$ for which $AC$ has fifth row $(0,0,1,0,0)$.  Since one can similarly do this for $B$, $[A]$ and $[B]$ must be orbit equivalent.

Next, assume that both $|a_{45}|^2 + |a_{55}|^2 = |b_{45}|^2 + |b_{55}|^2$ and $|\nu(A)| = |\nu(B)|$.  From Proposition \ref{prop:F}, $A$ and $B$ are orbit equivalent to matrices $C(\theta_A,\alpha_A)$ and $C(\theta_B, \alpha_B)$ in $\mathcal{F}$.  We note that $\cos(\theta_A)= |a_{45}|^2 + |a_{55}|^2 = |b_{45}|^2 + |b_{55}|^2 = \cos(\theta_B)$, so $\theta_A  =\theta_B$.  Similarly, we have $\sin(\theta_A)\sin(\alpha_A) = |\nu(A)| = |\nu(B)| = \sin(\theta_B)\sin(\alpha_B) = \sin(\theta_A)\sin(\alpha_B)$.  If $\sin(\theta_A) = 0$, then $\cos(\theta_A) = 1$, so $A$ and $B$ are equivalent by the first case above.  If $\sin(\theta_A)\neq 0$, we conclude that $\sin(\alpha_A) = \sin(\alpha_B)$, so that $\alpha_A = \alpha_B$.  Thus, in the case, we find that $A$ and $B$ are orbit equivalent, both being orbit equivalent to $C(\theta_A,\alpha_A)$.

\end{proof}

 \subsection{Metrics on Bazaikin spaces}\label{sec:bazmetrics}

We now apply the results of Section \ref{sec:wilking} to the case of Bazaikin spaces.  Our first proposition describes precisely which Bazaikin spaces admit a Wilking metric for which the natural isometry group acts with cohomogeneity two.

\begin{proposition}\label{prop:baz2}
Suppose $\overline{q} = (q_1,q_2,q_3,q_4,q_5)$ and that $\mathcal{B}_{\overline{q}}$ admits a Wilking metric for which the natural isometry group acts with cohomogeneity two.  Then, up to permutation, $\overline{q} = (q_1,q_1,q_1,q_4,q_4)$ with $|q_1 + q_4| = 2$.
\end{proposition}

\begin{proof}
Recall first that the natural isometry group is a subgroup of $L = (S\times \{I\})\cdot H'$.  In the proof of Proposition \ref{prop:cohom2}, we saw that if there are at least three distinct number among the $q_i$, then the $L$-action has cohomogeneity at least three.  It follows that the natural isometry group has cohomogeneity at least three in this case.

Thus, we may assume there are at most $2$ distinct numbers among the $q_i$.  Up to permutation, there are three cases to consider: $q_1$ can appear $5$ times, $4$ times, or $3$ times.

Assume initially that $q_1$ appears $5$ times so that all $5$ $q_i$ are equal.  Since the action is effectively free, we have $\gcd(q_1+q_1,q_1+q_1) = 2$ so that $|q_1| = 1$.  This then implies that $|q_1+q_4| = 2$.

Second, assume that $q_1$ appears $4$ times with the condition that $q_5\neq q_1$.  Then $S = S(U(4)\times U(1))\cong U(4)$.  Note that the largest proper subgroup of $S$ is isomorphic to $U(3)\times U(1)$, which has dimension $10$.  Since Bazaikin spaces are $13$-dimensional, there are no cohomogeneity two actions by any proper subgroup of $S$ on any Bazaikin space.  On the other hand, the action by $S$ itself is cohomogeneity one.  It follows that this case cannot arise.

Lastly, assume that $q_1$ appears $3$ times so that $q_4$ appears twice.  Then by Proposition \ref{prop:free}, $|q_1+q_4| = 2$.
\end{proof}

The case where all $q_i$ are equal gives rise to the Berger space.  So, for the rest of the section, we will assume $\overline{q} = (q_1,q_1,q_1,q_4,q_4)$ with $|q_1 + q_4|=2$ and $q_1\neq q_4$.  Thus, we also assume for the rest of this section that $S = S(U(3)\times U(2))$.  We wish to determine the set of possible Wilking metrics for which the natural isometry group acts on $\mathcal{B}_{\overline{q}}$ with cohomogeneity two.  Note that $S$ contains a unique connected subgroup $S' = SU(3)\times SU(2)$ of dimension $11$.  All other proper connected subgroups have dimension smaller than $11$, so cannot act via a cohomogeneity two action on $\mathcal{B}_{\overline{q}}$.  It follows that we must identify all the Wilking metrics for which $S$ or $S'$ act isometrically.

To begin, recall that $SU(5)$ is simple and hence, up to scaling, it admits a unique bi-invariant metric $\langle \cdot, \cdot\rangle_0$.  For definiteness, we take $\langle X,Y\rangle_0 = -ReTr(XY)$ for $X,Y \in \mathfrak{su}(5)$, the Lie algebra of $SU(5)$.

Next, recall that a Wilking metric is a submersion metric on $\Delta G \backslash (G\times G)/H'$ where $G\times G$ is equipped with a product of iterated Cheeger deformations of $\langle \cdot,\cdot\rangle_0$.  We begin by analyzing the possible metrics on the right factor of $G$.

So, suppose the metric on the right factor is obtained from $\langle \cdot,\cdot \rangle_0$ via an iterated Cheeger deformation along a chain of subgroups $\{I\}=K_{n+1}\subseteq K_n \subseteq ... \subseteq K_1\subseteq G$.  From Proposition \ref{prop:isom}, in order for the action of $H'$ on $G\times G$ to be isometric, we need $p_2(H')$ to normalize each $K_n$, where $p_2:G\times G\rightarrow G$ denotes projection onto the second factor.  In other words, we need $p_2(H')\subseteq N_G(K)$.

\begin{proposition}\label{prop:right}  Suppose $K\subseteq G = SU(5)$ is a connected compact Lie group.  Then $p_2(H')\subseteq N_G(K)$ if and only if $$K\in \{ \{I\}, Sp(2), S^1, p_2(H'), SU(4), U(4), SU(5)\}$$ where $Sp(2)$ and $S^1$ denote the simple factors of $p_2(H')$ and $U(4)$ is embedded into $G$ via $A\mapsto \diag(A,\overline{\det(A)})$.
\end{proposition}

\begin{proof}  First, assume $K$ is one of the listed elements in the set. 
 The first four are normal in $p_2(H)$, so that $p_2(H)\subseteq N_G(K)$. 
 For the last three, $p_2(H)\subseteq K \subseteq N_G(K)$.  This proves the ``if" direction.
 
 We now demonstrate the ``only if" direction.  We note that since $p_2(H')$ is connected, $p_2(H')\subseteq N_G(K)^0$, where $N_G(K)^0$ denotes the identity component of $N_G(K)$, so that $N_G(K)^0$ is a connected subgroup lying between $p_2(H')$ and $G$.  We claim that this implies $N_G(K)^0 \in \{ p_2(H'), U(4), G\}$.  To see this, we first claim that there is no simple group lying strictly between $p_2(H')$ and $G$.  Such a simple group would have to have rank $3$ or $4$ and dimension between $\dim p_2(H)+1 = 12$ and $\dim G-1 = 23$.  The complete list of simple groups (given up to cover) meeting these hypotheses are $Sp(3)$ and $Spin(7)$, and simple representation theory indicates that neither of these groups has an almost-faithful complex representation of dimension $5$ or less.

We now show that the only non-simple group lying between $p_2(H')$ and $G$ is $U(4)$. The maximal such group must appear in \cite[Table 5]{AFG} and contain a simple factor of rank at least $2$ and dimension at least $10$. Therefore, the only option is $U(4)$.

It remains to see that there are no (necessarily non-simple) intermediate subgroups between $p_2(H')$ and $U(4)$.  To that end, observe that the $Sp(2)$ factor of $p_2(H')$ is a maximal subgroup of the $SU(4)$ subgroup of $U(4)$ \cite[Proposition 2.3]{BLS}.  It follows from this that $p_2(H')$ is maximal in $U(4)$.

Having completed the proof that $N_G(K)^0\in \{p_2(H'), U(4), G\}$, we now consider each option and list the possibilities for $K$.  First, if $N_G(K)^0 = p_2(H')$, then $K$ must be a normal subgroup of $p_2(H')$.  Up to cover, $p_2(H')$ is isomorphic to $Sp(2)\times S^1$, so the only normal subgroups of $p_2(H')$ are $\{I\}, Sp(2), S^1,$ and $ p_2(H')$ with both $Sp(2)$ and $S^1$ normal in $p_2(H')$.

Second, if $N_G(K)^0 = U(4)$, then $K$ must be a normal subgroup of $U(4)$. As $U(4)$ is covered by $SU(4)\times S^1$, we find that $K$ is isomorphic to $\{I\}, SU(4), S^1$, or $U(4)$, and that $S^1$ must be normal in $U(4)$. 
 From the embedding of $U(4)$ in $G$, $S^1 = \{\diag(z,z,z,z,z^{-4})\}$, which is the normal $S^1$ in $p_2(H')$.

 Finally, if $N_G(K)^0 = SU(5)$, then $K$ must be a normal subgroup of $G=SU(5)$.  Since $G$ is simple, $K = \{I\} $ or $K= G$.

 Thus, among the three possibilities, we find that $$K \in \{ \{I\}, Sp(2), S^1, p_2(H'), SU(4), U(4), G\}$$ as claimed.

\end{proof}

We now focus on the possible Cheeger deformations on the left factor of $G\times G$. Since we desire for both the right multiplication by $p_1(H')$ and at least one of $S'$ or $S$ to be isometric, Proposition \ref{prop:isom} implies that we need both $p_1(H')$ and at least one of $S'$ or $S$ to normalize each $K_i$.  Since $p_1(H')\subseteq S'\subseteq S$, $p_1(H')$ will normalize each $K_i$ if at least one of $S'$ or $S$ does.

\begin{proposition}\label{prop:left}  Suppose $K\subseteq G = SU(5)$ is a connected compact Lie group.  Then the following are equivalent:

\begin{enumerate} \item $S\subseteq N_G(K)$

\item  $S'\subseteq N_G(K)$

\item  $K\in \{\{I\},S^1, SU(3), SU(2), S^1\cdot SU(3), S^1\cdot SU(2), S', S, G \},$ where $S^1$ denotes the matrices of the form $\diag(z^2,z^2,z^2, z^{-3},z^{-3})$, $SU(3)$ is embedded into $SU(5)$ as the top left $3\times 3$ block, and $SU(2)$ is embedded as the lower $2\times 2$ block.
\end{enumerate}
\end{proposition}

\begin{proof} We first argue that 3 implies 1 and 2.  To that end, simply note that for any possibility for $K$, either $K$ is normal subgroup of $S$ or $S$ is subgroup of $K$, so that $S\subseteq N_G(K)$.  Similarly, $K$ is a normal subgroup of $S'$ or $S'$ is a subgroup of $K$, so $S'\subseteq N_G(K)$.

We next argue that both 1 and 2 imply 3.  So, assume initially that $S\subseteq N_G(K)$.  From \cite{BdS}, $S$ is maximal among connected proper subgroups of $G$, so $N_G(K)^0 = S$ or $N_G(K)^0 = G$.  On the other hand, if $S'\subseteq N_G(K)$, then, since $S'$ is maximal among connected proper subgroups of $S$, we find that $N_G(K)^0\in \{S', S, G\}$.  Thus, both 1 and 2 each imply that $N_G(K)^0\in \{S',S,G\}$.  If $N_G(K)^0 = S'$ or $N_G(K)^0 = S$, it follows that $K$ is normal in $S$.  If $N_G(K)^0 = G$, then $K$ must be a normal subgroup of $G$.  Thus, in either case, we find $$K\in \{\{I\}, S^1, SU(3), SU(2), S^1\cdot SU(3), S^1\cdot SU(2), S', S, G\}$$ as claimed.

\end{proof}

It follows from Proposition \ref{prop:left} that for a given Wilking metric, $S$ acts isometrically if and only if $S'$ acts isometrically.   Thus, for the remainder of this section, we may assume $S$ acts isometrically.

Combining Propositions \ref{prop:right} and \ref{prop:left}, we have a complete determination of all of the Wilking metrics for which both the $H'$ and $S$ actions are isometric.  Specifically, such a metric is characterized as follows:  $G\times G$ is given a product metric $\langle \cdot, \cdot\rangle_\ell + \langle \cdot ,\cdot \rangle_r$ where $\langle \cdot, \cdot \rangle_\ell$ is obtained from $\langle \cdot, \cdot \rangle_0$ via an iterated Cheeger deformation along subgroups appearing in the statement of Proposition \ref{prop:left}, while $\langle \cdot, \cdot\rangle_r$ is obtained analogously using Proposition \ref{prop:right}.  However, it turns out that despite all the possibilities for these metrics, Theorem \ref{thm:main2} will follow once we prove it for one specific Wilking metric.

\begin{definition} The \textit{main Wilking metric} on a cohomogeneity two Bazaikin space is the Wilking metric for which the metric on the left factor is obtained as a single (non-iterated) Cheeger deformation in the direction of $S$ and the metric on the right factor is obtained a as a Cheeger deformation in the direction of $U(4)$.
\end{definition}

The importance of this metric is contained in the following proposition.

\begin{proposition}\label{prop:reducetomain}  Suppose $\mathcal{B}_{\overline{q}} = \Delta G\backslash(G\times G)/ H'$ is a cohomogeneity two Bazaikin space and consider a point $[g,e]\in \mathcal{B}_{\overline{q}}$.  If there is a zero-curvature plane at $[g,e]$ with respect to the main Wilking metric, then there is a zero-curvature plane with respect to any Wilking metric for which both the $H'$ and $S$-action are isometric.
\end{proposition}

\begin{proof}  Suppose there is a zero-curvature plane at $[g,e]$ with respect to the main Wilking metric.  Then, from Proposition \ref{prop:zerowilking}, we see that there are linearly independent vectors $X,Y\in \mathfrak{g}$ satisfying the three conditions of Proposition \ref{prop:zerowilking}.  From Proposition \ref{prop:curvature}, we see that condition $2$ is equivalent to the condition that \begin{equation}\label{eqn:curv1} [X,Y] = [X_{\mathfrak{u}(4)}, Y_{\mathfrak{u}(4)}] = 0.\end{equation}  Moreover condition $3$ is equivalent to the condition that \begin{equation}\label{eqn:curv2}[Ad_{g^{-1}}(X),Ad_{g^{-1}}(Y)] = [(Ad_{g^{-1}}(X))_{\mathfrak{s}}, (Ad_{g^{-1}}(Y))_{\mathfrak{s}}] = 0,\end{equation} where $\mathfrak{s}$ is the Lie algebra of $S$.

Now, consider any Wilking metric $\langle \cdot, \cdot\rangle$ on $\mathcal{B}_{\overline{q}}$ for which the $H'$ and $S$-actions are isometric.  We will show that $X$ and $Y$ satisfy the three conditions of Proposition \ref{prop:zerowilking} so that there is a zero-curvature plane at $[g,e]$ with respect to this metric.    Observe condition 1 only depends on $\langle \cdot, \cdot\rangle_0$, so is automatically verified.

To verify the other two conditions, we first set up notation.  Recall that, by definition, $\langle \cdot,\cdot\rangle$ is obtained as a submersion metric from $G\times G$, where $G\times G$ is equipped with a metric $\langle \cdot, \cdot\rangle_\ell + \langle \cdot, \cdot\rangle_r$, a product of iterated Cheeger metrics.  We let $\phi_\ell$ and $\phi_r$  denote the metric tensor corresponding to $\langle \cdot, \cdot \rangle_\ell$ and $\langle \cdot,\cdot\rangle_r$ respectively.  That is $\langle Z,W\rangle_\ell = \langle \phi_\ell(Z), W\rangle_0$ and $\langle Z,W\rangle_r = \langle \phi_r(Z),W\rangle_0$ for all $Z,W\in \mathfrak{g}$.

We now verify condition 2, that $\operatorname{span}\{\phi_r^{-1}(X), \phi_r^{-1}(Y)\}$ has zero-curvature.  Let $K$ denote any of the groups used in Cheeger deforming $\langle \cdot,\cdot\rangle_0$ to $\langle \cdot, \cdot \rangle_r$.  Then, from Proposition \ref{prop:right}, $$K\in \{ \{I\}, Sp(2), S^1, p_2(H'), SU(4), U(4), G\}.$$   Following Proposition \ref{prop:curvature}, we must verify that $[X_\mathfrak{k}, Y_\mathfrak{k}] = 0$ where $\mathfrak{k}$ is the Lie algebra of $K$.  If $K\in \{U(4), G\}$, this follows directly from \eqref{eqn:curv1}.  If $K = SU(4)$, this follows from \eqref{eqn:curv1} together with the fact that $\mathfrak{su}(4)$ is an ideal of $\mathfrak{u}(4)$.  The remaining possibilities for $K$ all have $K\subseteq p_2(H')$.  From condition 1, we know that $X$ and $Y$ have trivial projection to $\mathfrak{sp}(2)$ which implies that $X_{\mathfrak{k}}$ and $Y_{\mathfrak{k}}$ belong to the one-dimensional ideal in $p_2(H')$.  In particular, $[X_{\mathfrak{k}}, Y_{\mathfrak{k}}] = 0$.  This completes the verification of condition 2.

We finally turn to verifying condition 3, that $$\operatorname{span}\{ \phi_\ell^{-1}(Ad_g(X)), \phi_\ell^{-1}(Ad_g(Y))\}$$ has zero-curvature.  Let $K$ denote any of the groups used in Cheeger deforming $\langle \cdot,\cdot\rangle_0$ to $\langle \cdot, \cdot \rangle_\ell$.  Then, from Proposition \ref{prop:left}, $$K\in \{ \{I\}, S^1, SU(3), SU(2), S^1\cdot SU(3), S^1\cdot SU(2), S',S , G\}.$$  Following Proposition \ref{prop:curvature}, we must verify that $[(Ad_{g^{-1}}(X))_{\mathfrak{k}}, (Ad_{g^{-1}}(Y))_{\mathfrak{k}}]= 0$.  If $K\in \{L,G\}$, this follows directly from \eqref{eqn:curv2}.  The remaining possibilities for $K$ are all normal subgroups of $S$.  In particular, the condition that $[(Ad_{g^{-1}}(X))_{\mathfrak{s}}, (Ad_{g^{-1}}(Y))_{\mathfrak{s}}] = 0$  implies that $[(Ad_{g^{-1}}(X))_{\mathfrak{k}}, (Ad_{g^{-1}}(Y))_{\mathfrak{k}}] = 0$, completing the verification of condition 3.

\end{proof}

\section{Zero-curvature planes in the main Wilking metric}\label{sec:zero}

The goal of this section is to prove Theorem \ref{thm:main2}.  In light of Propositions \ref{prop:baz2} and \ref{prop:reducetomain}, to prove Theorem \ref{thm:main2}, it is sufficient to verify that the main Wilking metric is not almost positively curved on any $\mathcal{B}_{\overline{q}}$ with $\overline{q} = (q_1,q_1,q_1,q_4,q_4)$ with $|q_1 + q_4 | = 2$ and $q_1\neq q_4$.  That is, we must find an open set of points in $\mathcal{B}_{\overline{q}}$ having at least one zero-curvature plane.   Because replacing all $q_i$ with their negatives yields isometric Bazaikin spaces, we will assume that $q_1  > 0$.   We set $q_1 + q_4 = 2\omega$ with $\omega \in \{\pm 1 \}$.  In addition, we assume throughout that $(q_1,\omega)\neq (1,1)$, so that $q_1 - \omega \neq 0$.

We will prove the following proposition.

\begin{proposition}\label{prop:zerocurvF} For each $q_1 >0$ and $\omega \in \{\pm 1\}$ except $(q_1,\omega) = (1,1)$, there is a non-empty open subset $U\subseteq \mathcal{F}$ with the property that for every $A=A(\theta,\alpha)\in U$, $A$ projects to a point in $\mathcal{B}_{\overline{q}}$ having at least one zero-curvature plane.
\end{proposition}

Before proving Proposition \ref{prop:zerocurvF}, we will show that it leads to a proof of Theorem \ref{thm:main2}.

\begin{proposition} If Proposition \ref{prop:zerocurvF} is true, then $\mathcal{B}_{\overline{q}}$ is not almost positively curved.
\end{proposition}

\begin{proof}Suppose $U$ is as in Proposition \ref{prop:zerocurvF}.    Recall that $\mathcal{F}$ is homeomorphic to a closed two-dimensional disk.  By intersecting $U$ with the interior of $\mathcal{F}$, we may replace $U$ with this smaller set if necessary, and hence we may assume that $U$ lies entirely in the interior of $\mathcal{F}.$ We claim that the $S$-orbit $S\cdot U$ is open in $G$. 

Consider the function $f:SU(5)\rightarrow \mathbb{R}$ defined by $f(B) = |b_{45}|^2 + |b_{55}|^2$ and recall the function $\nu$ from Proposition \ref{prop:ABequivalent}.  We combine these into a continuous function $g:=(f,|\nu|):G\rightarrow\mathbb{R}^2$ mapping $B$ to $(|b_{45}|^2 + |b_{55}|^2, |\nu(B)|)$.

For $A = (a_{ii})\in U$, we compute that $g(A) = (\cos^2(\theta),\sin(\theta)\sin(\alpha))$.  Since $U$ lies in the interior of $\mathcal{F}$, it now easily follows that $g|_U:U\rightarrow \mathbb{R}^2$ is injective.  By invariance of domain, we conclude that $g(U)$ is an open subset of $\mathbb{R}^2$.

Then $g^{-1}(g(U))$ is an open subset of $G = SU(5)$.  By Proposition \ref{prop:ABequivalent}, we see that $S\cdot U = g^{-1}(g(U))$, so that $S\cdot U$ is an open subset of $G$.

Now, since the action of $S\times \{I\}$ on $G\times G$ is isometric, it follows that the induced $S$-action on $\Delta G\backslash (G\times G)\cong G$ is isometric, so that $S\cdot U$ is an open set of points in $G$, all of which project to points in $\mathcal{B}_{\overline{q}}$ having at least one zero-curvature plane.  Since the projection map $G\rightarrow \mathcal{B}_{\overline{q}}$ is a submersion, it is an open map, so we have identified a non-empty open subset of $\mathcal{B}_{\overline{q}}$ consisting of points with zero-curvature planes.
\end{proof}

For the duration of this section, we will set $$X:=\begin{bmatrix} x_{11} & x_{12} &  0 & 0 & 0 \\ -\overline{x}_{12} & x_{22} & 0 &0 & 0\\ 0 & 0 & x_{11} & -\overline{x}_{12} & 0\\ 0 & 0 & x_{12} & x_{22} & 0\\ 0 & 0 & 0 & 0 & -2x_{11} - 2x_{22}\end{bmatrix}$$  and $$ Y:= \begin{bmatrix} \lambda i &0 & 0 & 0 & i\\ 0 & \lambda i & 0 & 0 & \overline{x}_{12}\\ 0 & 0 & \lambda i &0 & 0\\ 0& 0 & 0& \lambda i& 0\\ i & - x_{12} &0 & 0 & -4\lambda i\end{bmatrix},$$ where the variables $x_{11}, x_{22}\in \IM(\mathbb{H})$, $\lambda \in \mathbb{R}$, $x_{12}\in \mathbb{H}$ are arbitrary.  We observe that both $X,Y\in \mathfrak{g}$.  To simplify later formulas, we define real numbers $z_1$ and $z_2$ as $$iz_1 = 3x_{11} + 2x_{22}\text{ and } iz_2 = 2x_{11} + 3x_{22}$$ so that$$ x_{11} = \frac{1}{5}i (3z_1 - 2z_2) \text{ and } x_{22} = -\frac{1}{5}i(2z_1 - 3z_2).$$ 

Now, using the notation from Section \ref{sec:wilking}, let $$ \hat{X}=( -\phi_{\ell}^{-1}(Ad_{g^{-1}}(X)), \phi_r^{-1}(X))$$ and $$\hat{Y} = ( -\phi_{\ell}^{-1}(Ad_{g^{-1}}(Y)), \phi_r^{-1}(Y)) $$ and let $\sigma = \operatorname{span}\{\hat{X},\hat{Y}\}.$

\begin{proposition}\label{prop:getzero} Suppose $A=A(\theta,\alpha)\in \mathcal{F}$.  Then, the left translation of $\sigma$ to the point $[A^{-1},e]\in \Delta G\backslash(G\times G)/H'$ has zero-curvature with respect to the main Wilking metric if all of the following conditions are satisfied

\begin{multline}\label{aeqn:xtrace}(z_1 - z_2)\omega\cos(\alpha)^2 + 2\IM(x_{12})(q_1 - \omega)\cos(\alpha)\sin(\alpha)\\ + q_1(z_1-z_2)\sin(\alpha)^2 + z_1((\omega -q_1)  - 2\omega)=0\end{multline}

\begin{equation}\label{aeqn:ytrace} -5\lambda(q_1 - \omega)\cos(\theta)^2 + 2(q_1 - \omega)\sin(\theta)\cos(\theta) - 10\omega\lambda=0\end{equation}

\begin{equation}\label{aeqn:brack1} - z_1 + |x_{12}|^2 = 0\end{equation}

\begin{equation}\label{aeqn:brack2} \overline{x}_{12}(z_2 - 1)=0\end{equation}

\begin{equation}\label{aeqn:Ad}i \overline{x}_{12}\sin(\theta)((z_1 + 2)\cos(\theta)^2 + 5\lambda\cos(\theta)\sin(\theta) - z_2)=0 \end{equation}

\end{proposition}

\begin{proof}We will show that these $5$ equations imply all three conditions of Proposition \ref{prop:zerowilking}.  We first observe that with our definition of $H'$ above, $\mathfrak{h}'$ is spanned by a vector of the form $$(\diag( 5q_1-q, 5q_1-q, 5q_1-q,5q_4-q, 5q_4-q)i, (-q, -q, -q, -q, 4q)i)\in \mathfrak{g}\oplus \mathfrak{g}$$ (with $q_1 + q_4 = 2\omega$) and $0\oplus \mathfrak{sp}(2)\subseteq \mathfrak{g}\oplus \mathfrak{g}$ where $$\mathfrak{sp}(2) = \left\{ \begin{bmatrix} iw_1  & u_{12} & u_{13} & u_{14} & 0 \\
-\overline{u}_{12} & i w_2  & u_{14} & u_{24} & 0\\
-\overline{u}_{13} & -\overline{u}_{14} & -i w_1 & \overline{u}_{12} & 0\\
-\overline{u}_{14} & -\overline{u}_{24} & -u_{12} & -i w_2  & 0\\ 0 & 0 & 0 & 0 & 0\end{bmatrix}: w_i \in\mathbb{R}, u_{ij} \in \mathbb{C} \right\}.$$   Using this, it is routine to verify that equations \eqref{aeqn:xtrace} and \eqref{aeqn:ytrace} imply condition $1$ of Proposition \ref{prop:zerowilking}.

Towards verifying condition 2, we must verify that both $[X,Y] = 0$ and that $[X_{\mathfrak{u}(4)}, Y_{\mathfrak{u}(4)}] = 0$.  Equations \eqref{aeqn:brack1} and \eqref{aeqn:brack2} imply that $[X,Y]=0$ and $Y_{\mathfrak{u}(4)}$ lies in the center of $\mathfrak{u}(4)$, so has trivial bracket with anything in $\mathfrak{u}(4)$.

Finally, we verify condition 3.  We need to check that both $$[Ad_{A^{-1}}(X), Ad_{A^{-1}}(Y)] = 0 \text{ and } [(Ad_{A^{-1}}(X))_{\mathfrak{s}}, (Ad_{A^{-1}}(Y))_{\mathfrak{s}}]=0.$$  The first is automatic  since we have already verified that $[X,Y] = 0$ and $Ad_{A^{-1}}$ is a Lie algebra isomorphism.  For the second Lie bracket, we observe that by our choice of $X$ and $Y$, all entries of $[(Ad_{A^{-1}}(X))_{\mathfrak{s}}, (Ad_{A^{-1}}(Y))_{\mathfrak{s}}]$ vanish except for the $(1,3)$ and $(3,1)$ entries.  The vanishing of these entries is equivalent to equation \eqref{aeqn:Ad}, which completes the proof.
\end{proof}

Thus, to prove Proposition \ref{prop:zerocurvF}, we need to find a non-empty open set $U\subseteq \mathcal{F}$ for which we can find $\lambda ,z_1, z_2\in \mathbb{R}$ and $x_{12}\in \mathbb{C}$ solving all the equations of Proposition \ref{prop:getzero}.  In fact, we now give the formulas for these variables:

\begin{itemize} \item  $\displaystyle \lambda = \frac{2\cos(\theta)\sin(\theta)(\omega-q_1)}{5(\cos^2(\theta)\omega-\cos^2(\theta)q_1-2\omega)} $ 

\item  $ \displaystyle z_1 = -\frac{5\cos(\theta)\sin(\theta) \lambda +2\cos^2(\theta)-1}{\cos^2(\theta)}$

\item  $z_2 = 1$

\item $x_{12} = Re(x_{12}) + \IM(x_{12})i$ where $\IM(x_{12})$ is given by  $$ \frac{z_1( -\cos^2(\alpha)\omega - \sin^2(\alpha)q_1  +(q_1-\omega) \cos^2(\theta) + 2\omega) + \cos^2(\alpha)\omega + \sin^2(\alpha)q_1}{2\cos(\alpha)\sin(\alpha)(q_1-\omega)}$$  and $$Re(x_{12}) =  \sqrt{z_1 - \IM(x_{12})^2}.$$

\end{itemize}

It is clear from the formulas that $\lambda, z_1, z_2$, and $\IM(x_{12})$ are real numbers whenever they have non-vanishing denominators.  On the other hand, when $(q_1,\omega)\neq (1,-1)$, one easily sees that $z_1 = -1$ when $\theta = 0$.  Thus, $z_1 < 0$ so that our formula for $Re(x_{12})$ produces an imaginary number.  Nonetheless, we will later prove that there is a non-empty open set $V\subseteq (0,\pi/2)\times (0,\pi/2)$ for which $Re(x_{12})$ is actually a real number.

Before doing so, we note the following proposition, whose verification is straightforward.

\begin{proposition}\label{prop:verify}  Suppose $Re(x_{12})$ is real.  Then with the above definitions of $\lambda_0,z_1,z_2, $ and $x_{12}$, all the equations of Proposition \ref{prop:getzero} are satisfied.
\end{proposition}

Our next goal is to find a non-empty open set $V\subseteq (0,\pi/2)\times (0,\pi/2)$ for which $Re(x_{12})$ is real.  As noted above, the case where $(q_1,\omega) = (1,-1)$ works differently, so we will assume for now that $q_1 > 0$.  Later, we will complete the proof in the case $(q_1,\omega) = (1,-1)$.

We define $\theta_0$ and $\alpha_0 = \alpha_0(\theta)$ by the formulas $$\cos^2(\theta_0) =\frac{2}{q_1+1} \text{ and }\cos^2(\alpha_0) = \frac{(q_1-\omega)\sin^2(\theta) -\omega}{q_1-\omega}.$$  Since $q_1 >1$, we see that $0 < \frac{2}{q_1+1}<1$ so this defines $\theta_0\in (0,\pi/2)$ uniquely.  We now show an analogous result for $\alpha_0$, with appropriate restrictions on $\theta$.

\begin{lem}\label{lem:alphabound}  When $\omega = -1$, we have $0 < \frac{(q_1+1)\sin^2(\theta)+1}{q_1+1} < 1$ for  all $\theta>\theta_0$  sufficiently close to $\theta_0$.  When $\omega = 1$, we have $0< \frac{(q_1-1)\sin^2(\theta) - 1}{q_1-1} < 1$ when $\theta > \arcsin\sqrt{\frac{1}{q_1-1}}$.
\end{lem}

\begin{proof}  We begin with the case where $\omega = -1$.  By continuity, it is sufficient to establish the inequality when $\theta = \theta_0$.  In this case, we have $$\sin^2(\theta_0) = 1-\cos^2(\theta_0) = 1-\frac{2}{q_1+1} = \frac{q_1-1}{q_1+1}.$$  Substituting this in, we find $$\frac{(q_1+1)\sin^2(\theta)+1}{q_1+1} = \frac{(q_1+1) \frac{q_1-1}{q_1+1} + 1}{q_1+1} = \frac{q_1}{q_1+1}, $$ and now the desired inequality is obvious.

\bigskip

We now turn to the case where $\omega = 1$.  In this case, one finds that $0 < \frac{(q_1-1)\sin^2(\theta)-1}{q_1-1}$ if and only if $\frac{1}{q_1-1} < \sin^2(\theta)$ which holds if and only if $\arcsin\sqrt{\frac{1}{q_1-1}}<\theta$.  In addition, the inequality $\frac{(q_1-1)\sin^2(\theta)-1}{q_1-1} < 1$ holds if and only if $\sin^2(\theta) < \frac{q_1}{q_1-1}$.  But, as $\frac{q_1}{q_1-1} > 1$, the latter inequality holds for all $\theta$.
\end{proof}

We need another lemma.

\begin{lem}\label{lem:theta}  When $\omega = 1$, the denominator of $\lambda$ never vanishes.  When $\omega = -1$, $\theta_0$ is the unique element of $(0,\pi/2)$ for which the denominator of $\lambda$ vanishes.  For either choice of $\omega$, the number $\alpha_0$ has the property that the coefficient of $z_1$ in the numerator of $\IM(x_{12})$ vanishes.

\end{lem}

\begin{proof}  When $\omega = 1$, the denominator of $\lambda$ is $5  (\cos^2(\theta)(1-q_1) - 2)$.  The factor in parenthesis is a sum of a non-positive term and a negative term, so it is negative.  When $\omega = -1$, the claim for $\theta_0$ is obvious and follows simply by solving $\cos^2(\theta)(\omega - q_1) - 2\omega = 0$.

For the last claim, we focus on the coefficient of $z_1$:  $$-\cos^2(\alpha) -\sin^2(\alpha)q_1 + (q_1-\omega)\cos^2(\theta) + 2\omega).$$  We first observe that $$\sin^2(\alpha_0) = 1-\cos^2(\alpha_0) = \frac{(q_1-\omega)\cos^2(\theta)+\omega}{q_1-\omega}.$$

Then \begin{align*}  &-\cos^2(\alpha_0)\omega - \sin^2(\alpha) q_1 + (q_1-\omega)\cos^2(\theta)+2\omega\\
&= -\frac{(q_1-\omega)\sin^2(\theta)-\omega}{q_1-\omega} \omega - \frac{(q_1-\omega)\cos^2(\theta) +\omega}{q_1-\omega} q_1 + (q_1-\omega)\cos^2(\theta) + 2\omega\\
&= -\frac{(q_1-\omega)\sin^2(\theta)\omega -\omega^2 + (q_1-\omega)\cos^2(\theta)q_1 + \omega q_1}{q_1-\omega} + (q_1-\omega)\cos^2(\theta) + 2\omega\\
&= -\frac{(q_1-\omega)(\sin^2(\theta)\omega + \cos^2(\theta)q_1 +\omega)}{q_1-\omega}+ (q_1-\omega)\cos^2(\theta) + 2\omega\\
&= -\sin^2(\theta)\omega - \cos^2(\theta)q_1 - \omega + q_1\cos^2(\theta) -\omega\cos^2(\theta) + 2\omega\\
&= \omega(-\sin^2(\theta)-\cos^2(\theta) -1 + 2)\\
&= 0.\end{align*}

\end{proof}

We now show that $Re(x_{12})$ is real on an open subset $V\subseteq (0,\pi/2)\times (0,\pi/2)$, beginning with the case where $q_1 > 0$.

\begin{proposition}\label{prop:Vopen}  For each integer $q_1 > 1$ and each $\omega \in \{\pm 1\}$, there is a non-empty open set $V\subseteq (0,\pi/2)\times (0,\pi/2)$ with the property that $z_1 - \IM(x_{12})^2 >0$ for every $(\theta,\alpha)\in V$.

\end{proposition}  

\begin{proof}  From Lemma \ref{lem:theta}, it follows that for either choice of $\omega$ that $z_1 - \IM(x_{12})^2$ is, as a function of $(\theta,\alpha)$, continuous on $(\theta_0,\pi/2)\times (0,\pi/2)$.  By continuity, it is sufficient to find one point for which $z_1 - \IM(x_{12})^2 > 0$.  We will find such a point by finding a path $(\theta,\alpha)$ for which $z_1$ goes to $\infty$ along this path while $\IM(x_{12})$ stays bounded.  It will then follow that for some point on this path, that $z_1 - \IM(x_{12})^2 > 0$.  We will break into cases depending on whether $\omega = -1$ or $\omega = 1$.

\textbf{Case 1:  } $\omega = -1$

In this case, we use the path $(\theta, \alpha_0(\theta))$ where $\theta\rightarrow \theta_0$ from the right.  By Lemma \ref{lem:alphabound}, $\alpha_0(\theta)$ is well-defined once $\theta$ is close enough to $\theta_0$.  We now show that $z_1\rightarrow \infty$ along this path while $\IM(x_{12})$ stays bounded.

To see that $\lim_{\theta\rightarrow\theta_0^+} z_1 = \infty$, we first observe that $\lim_{\theta\rightarrow \theta_0^+} \lambda  = -\infty$.  This follows since the numerator of $\lambda$ is obviously bounded above by $\omega - q_1 < 0$ while the denominator is positive on $(\theta_0, \pi/2)$, with limiting value $0$ at $\theta = \theta_0$.

It now easily follows that $\lim_{\theta\rightarrow \theta_0} z_1 = \infty$, independent of $\alpha$.

\bigskip

Next we will show that $\IM(x_{12})$ is bounded along this path.  To begin, observe that from Lemma \ref{lem:theta}, $$\IM(x_{12})|_{\alpha = \alpha_0} = \frac{-\cos^2(\alpha_0) + \sin^2(\alpha_0) q_1}{2\cos(\alpha_0)\sin(\alpha_0)(q_1 - \omega)}.$$  In particular, both the numerator and denominator are bounded near $\theta_0$.

We now bound the denominator away from $0$.  From the proof of Lemma \ref{lem:alphabound}, we see that $\lim_{\theta\rightarrow \theta_0} \cos(\alpha_0) = \sqrt{\frac{q_1}{q_1+1}}\neq 0$.  It follows that $$\lim_{\theta\rightarrow \theta_0} \sin(\alpha_0) = \sqrt{\frac{1}{q_1+1}}\neq 0.$$  Thus, the denominator of $\IM(x_{12})|_{\alpha = \alpha_0}$ has a non-zero limit at $\theta = \theta_0$.  It follows from all this that $\IM(x_{12})_{\alpha = \alpha_0}$ is bounded as $\theta\rightarrow \theta_0$.  This completes the proof in the case $\omega = -1$.

\bigskip

\textbf{Case 2}: $\omega = 1$.

In this case, we use the path $(\theta, \alpha_0(\theta))$ where $\theta\rightarrow \pi/2$ from the left.  By Lemma \ref{lem:alphabound}, $\alpha_0(\theta)$ is well-defined once $\theta > \arcsin \sqrt{\frac{1}{q_1-1}}$ and, in particular, $\alpha_0(\theta)$ is well-defined when $\theta$ is sufficiently close to $\pi/2$.  We now show that $z_1\rightarrow \infty$ along this path while $\IM(x_{12})$ stays bounded.

Towards computing $\lim_{\theta\rightarrow (\pi/2)^{-}} z_1$, we first observe that $$\lim_{\theta\rightarrow (\pi/2)^-} \frac{\lambda}{\cos(\theta)} = \frac{\omega - q_1}{-5\omega}$$ which is bounded.  It follows that as $\theta\rightarrow (\pi/2)$ from the left, that the term $\frac{1}{\cos^2(\theta)}$ of $z_1$ dominates.  In particular, $\lim_{\theta\rightarrow (\pi/2)^-} z_1 = \infty$.

We now show that $\IM(x_{12})_{\alpha = \alpha_0}$ stays bounded as $\theta$ approaches $\pi/2$ from the left.  To see this, first observe that $\lim_{\theta\rightarrow (\pi/2)^-} \cos^2(\alpha_0) = \frac{q_1-2}{q_1 - 1}\neq 0$.  Similarly, one finds that $$\lim_{\theta\rightarrow (\pi/2)^-} \sin^2(\alpha_0) = \frac{1}{q_1-1}\neq 0.$$  Thus, the denominator of $\IM(x_{12})_{\alpha = \alpha_0}$ is bounded away from $0$ as $\theta\rightarrow \pi/2$.  Since the numerator is obviously bounded, the proof is complete in this case as well.

\end{proof}

We now turn attention to establishing the existence of a $V\subseteq (0,\pi/2)\times (0,\pi/2)$ for which $z_1 - \IM(x_{12})^2 > 0$ when $(q_1,\omega) = (1,-1)$.

\begin{proposition}\label{prop:q1=1} When $(q_1,\omega) = (1,-1)$, there is a non-empty open set $V\subseteq (0,\pi/2)\times (0,\pi/2)$ with the property that $Z_1 - \IM(x_{12})^2 > 0$ for every $(\theta,\alpha)\in V$.
\end{proposition}

\begin{proof} As in the proof of Proposition \ref{prop:Vopen}, by continuity it is sufficient to find one point $(\theta_0,\alpha_0)$ for which $z_1 - \IM(x_{12})^2 > 0$.

In fact, one can take $(\theta_0,\alpha_0) = (\pi/4,\pi/4)$: at this point, one finds that $\lambda = -\frac{2}{5}$, $z_1 = 2,$ and $\IM(x_{12}) = -1$, so $z_1 - \IM(x_{12})^2 = 2-1 = 1 > 0$.

\end{proof}

We are now in a position to prove Proposition \ref{prop:zerocurvF}.

\begin{proof}(Proof of Proposition \ref{prop:zerocurvF})  From Propositions \ref{prop:Vopen} and \ref{prop:q1=1}, we know that there is a non-empty open set $V\subseteq (0,\pi/2)\times (0,\pi/2)$ for which $z_1 - \IM(x_{12})^2 > 0$ on $V$.

We let $U\subseteq \mathcal{F}$ be defined by $$U = \{ A(\theta,\alpha): (\theta,\alpha)\in V\}.$$  Since $V$ is a non-empty open subset of $(0,\pi/2)\times (0,\pi/2)$, $U$ is a non-empty open subset of $\mathcal{F}$.  We now show that for every $A(\theta,\alpha)\in U$, that there is at least one zero-curvature plane at $[A]\in \mathcal{B}_{\overline{q}}$.

To show this, recall that we have an isometry $$ \mathcal{B}_{\overline{q}} \cong \Delta G \backslash (G\times G)/H'$$ where $[A]$ maps to $[A^{-1}, e]$.

Since $z_1 - (\IM(x_{12})^2 > 0$, we can form the matrices $X$ and $Y$ as above.  From Propositions \ref{prop:getzero} and \ref{prop:verify}, there is a zero-curvature plane at $[A^{-1},e]$, as claimed.
\end{proof}

\bibliographystyle{plain}
\bibliography{bibliography.bib}

\end{document}